\documentclass[english]{article}
\usepackage[margin=2.5cm]{geometry}
\usepackage[hyphens]{url}
\usepackage{hyperref}
\usepackage[hyphenbreaks]{breakurl}
\usepackage{amsthm,amsfonts,amssymb,amsmath}
\usepackage{graphicx}
\usepackage{bbm}

\newtheorem{theorem}{Theorem}[section]
\newtheorem{proposition}[theorem]{Proposition}
\newtheorem{lemma}[theorem]{Lemma}

\newtheorem{claim}[theorem]{Claim}

\newtheorem{question}[theorem]{Question}

\theoremstyle{definition}

\numberwithin{equation}{section}
\numberwithin{figure}{section}
\allowdisplaybreaks

\usepackage{babel}
\usepackage[T1]{fontenc}
\usepackage[utf8]{inputenc}

\newcommand{\eps}{\varepsilon}
\newcommand{\su}{\subseteq}

\newcommand{\F}{\mathbb{F}}
\newcommand{\Fpn}{\mathbb{F}_p^n}
\newcommand{\Z}{\mathbb{Z}}

\newcommand{\EE}{\mathbb{E}}
\newcommand{\s}{\mathfrak{s}}

\begin{document}
\title{\vspace{-1.3cm}
On the Erd\H{o}s--Ginzburg--Ziv Problem in large dimension}
\author{Lisa Sauermann\thanks{Department of Mathematics, Massachusetts Institute of Technology, Cambridge, MA.
Email: \href{lsauerma@mit.edu}{\nolinkurl{lsauerma@mit.edu}}. Research supported by NSF Award DMS-2100157 and a Sloan Research Fellowship.}\ \and Dmitrii Zakharov\thanks{Department of Mathematics, Massachusetts Institute of Technology, Cambridge, MA.
Email: \href{zakhdm@mit.edu}{\nolinkurl{zakhdm@mit.edu}}.}}
\date{}

\maketitle

\vspace{-0.5cm}

\begin{abstract}\noindent
The Erd\H{o}s--Ginzburg--Ziv Problem is a classical extremal problem in discrete geometry. Given $m$ and $n$, the problem asks about the smallest number $s$ such that among any $s$ points in the integer lattice $\mathbb{Z}^n$ one can find $m$ points whose centroid is again a lattice point. Despite of a lot of attention over the last 50 years, this problem is far from well-understood. For fixed dimension $n$, Alon and Dubiner proved that the answer grows linearly with $m$. In this paper, we focus on the opposite case, where the number $m$ is fixed and the dimension $n$ is large. We drastically improve the previous upper bounds in this regime, showing that for every $\varepsilon>0$ the answer is at most $D_{\varepsilon,m}\cdot (C_{\varepsilon}m^{\varepsilon})^n$ for all $m$ and $n$. Our proof combines (a consequence of) the slice rank polynomial method with a higher-uniformity version of the Balog--Szemer\'{e}di--Gowers Theorem due to Borenstein and Croot.
\end{abstract}

\section{Introduction}

For given positive integers $m$ and $n$, what is the minimum number $s$ such that among any $s$ points in the $n$-dimensional integer lattice $\mathbb{Z}^n$ one can always find $m$ points whose centroid is again a lattice point in $\mathbb{Z}^n$? This problem is called the Erd\H{o}s--Ginzburg--Ziv Problem, and its answer is denoted by $\s(\mathbb{Z}_m^n)$ and called the \emph{Erd\H{o}s--Ginzburg--Ziv constant} of $\mathbb{Z}_m^n$. This notation reflects that the problem can be naturally translated into $\mathbb{Z}_m^n$ (where one then asks about the smallest $s$ such that any sequence of length $s$ of elements of $\mathbb{Z}_m^n$ contains a subsequence of length $m$ summing to zero).


This problem has been studied for fifty years (see e.g.\ \cite{harborth, kemnitz}), and is still wide open despite of receiving a lot of attention (in particular, over the past five years more than twenty papers were published on this topic). Only very few values of $\s(\mathbb{Z}_m^n)$ are known exactly: For $n=1$ Erd\H{o}s, Ginzburg, and Ziv \cite{egz} proved that $\s(\mathbb{Z}_m^1)=2m-1$, and for $n=2$ Reiher \cite{reiher} proved that $\s(\mathbb{Z}_m^2)=4m-3$. The only other infinite family of known values is when $m$ is a power of $2$, then $\s(\mathbb{Z}_{m}^n)=(m-1)\cdot 2^n+1$ as established by Harborth~\cite{harborth}.

Furthermore, and maybe more importantly, the growth behaviour of the function $\s(\mathbb{Z}_m^n)$ is far from being understood. For fixed dimension $n$, Alon and Dubiner \cite{alon-dubiner} proved that $\s(\mathbb{Z}_m^n)$ grows linearly with $m$. Improving their bound for the linearity constant, the second author \cite{dima-egz} furthermore showed that the bound $\s(\mathbb{Z}_m^n)\le 4^n\cdot m$  holds for every positive integer $m$ all of whose prime factors are sufficiently large with respect to $n$.

However, in the opposite regime for fixed $m$, there was an enormous gap between the upper and lower bounds. It turns out that in this regime, the problem can essentially be reduced (up to constant factors) to the case where $m$ is a prime. By the pigeonhole principle, one can easily obtain an upper bound of $\s(\mathbb{Z}_m^n)\le m\cdot m^n$ (as first observed by Harborth \cite{harborth}). This trivial bound was improved  to an upper bound of the form $C_m\cdot (\Gamma_m)^n$ for every fixed prime $m\ge 5$ by Naslund \cite{naslund} and for $m=3$ by Ellenberg--Gijswijt \cite{ellenberg-gijswijt}, where $C_m$ and $\Gamma_m$ are constants only depending on $m$ with $0.84m \le \Gamma_m\le 0.92m$. Note that the base $\Gamma_m$ of the main term $(\Gamma_m)^n$ in Naslund's bound is smaller than the base $m$ in the corresponding term $m^n$ in the trivial bound, but the base $\Gamma_m$ is still linear in $m$. This was improved by the first author \cite{lisa-egz}, who showed a bound of the form $C_m\cdot (2\sqrt{m})^n$ for every fixed prime $m\ge 5$, where again $C_m$ is a constant only depending on $m$. Here, the base is of the form $m^{1/2+o(1)}$, whereas the bases for the previous bounds were of the form $m^{1-o(1)}$ with $m$ being the base in the trivial bound. As discussed below, the $\sqrt{m}$ term in the base in this bound constitutes an important barrier for this problem.

Breaking this barrier, we drastically improve these upper bounds to an upper bound with base $m^{o(1)}$. More precisely, we show the following theorem bounding $\s(\mathbb{Z}_m^n)$ for any fixed integer $m$ and large dimension $n$.

\begin{theorem}\label{thm-egz-integer-m}
    For every fixed $\eps>0$ and every fixed integer $m\ge 2$, we have $\s(\mathbb{Z}_m^n)\le D_{\eps,m}\cdot (C_\eps m^{\eps})^n$ for all $n$. Here, $C_\eps$ is a constant only depending on $\eps$, and $D_{\eps,m}$ is a constant only depending on $\eps$ and $m$.
\end{theorem}

As mentioned above, the problem of upper-bounding $\s(\mathbb{Z}_m^n)$ for fixed $m$ and large $n$ can easily be reduced to the case where $m=p$ is a prime. The problem is then essentially equivalent (up to constant factors depending on $m=p$) to the following additive combinatorics problem: For a fixed prime $p$ and large $n$, what is the maximum possible size of a subset of $\Fpn$ not containing $p$ distinct vectors with sum zero? Our upper bound for $\s(\mathbb{Z}_m^n)$  in Theorem \ref{thm-egz-integer-m} is obtained by proving the following new upper bound for this additive combinatorics problem.

\begin{theorem}\label{thm-main}
For every fixed $\eps>0$ and every fixed prime $p$, the following holds for all $n$. For any subset $A\su \Fpn$ not containing distinct vectors $x_1,\dots,x_p\in A$ with $x_1+\dots+x_p=0$, we have $|A|\le D_{\eps,p}\cdot (C_\eps p^{\eps})^n$. Here, $C_\eps$ is a constant only depending on $\eps$, and $D_{\eps,p}$ is a constant only depending on $\eps$ and $p$.
\end{theorem}

The previous bounds for $\s(\mathbb{Z}_m^n)$ for fixed $m$ and large $n$ in \cite{naslund,lisa-egz} have also been obtained by studying this additive combinatorics problem, but here we prove much stronger bounds for this problem and hence for $\s(\mathbb{Z}_m^n)$. There is extensive literature and research activity on zero-sum problems in abelian groups (e.g.\ see the survey \cite{gao-geroldinger}), and this additive combinatorics problem is one of the most central problems in this area.

In the case of $p=3$, a subset $A\su \F_3^n$ not containing three distinct vectors $x_1,x_2,x_3\in A$ with $x_1+x_2+x_3=0$ is precisely the same as a three-term progression-free subset $A\su \F_3^n$. The problem of determining the maximum possible size of such a subset is a very famous problem in additive combinatorics, called the \emph{cap-set problem}. In 2017, Ellenberg and Gijsiwjt \cite{ellenberg-gijswijt} achieved a breakthrough on this problem, proving that any subset $A\su \F_3^n$ without a three-term arithmetic progression has size at most $2.756^n$. Hence for $p=3$, in Theorem \ref{thm-main} one has the bound $|A|\le 2.756^n$. 

Ellenberg and Gijsiwjt \cite{ellenberg-gijswijt}  actually proved a more general result, bounding the size of a three-term progression-free subset $A\su \F_p^n$ for any fixed prime $p\ge 3$. Their proof relies on a new polynomial method that was introduced by Croot, Lev and Pach \cite{croot-lev-pach} just a few weeks earlier, and that was shortly afterwards reformulated and generalized by Tao \cite{tao} to what is now called the \emph{slice rank polynomial method}. Since the slice rank polynomial gives the bound $|A|\le 2.756^n$ in Theorem \ref{thm-main} for $p=3$, it is very natural to also try to apply it in the setting of Theorem \ref{thm-main} for larger primes $p$. Unfortunately, the proof for $p=3$ breaks down for larger $p$, since the slice rank polynomial method cannot handle the distinctness condition for the vectors $x_1,\dots,x_p$ in Theorem \ref{thm-main}. This is because, with this distinctness condition, the relevant tensor is not a diagonal tensor anymore, and so one looses control over its slice rank.

The above-mentioned works of Naslund \cite{naslund} and the first author \cite{lisa-egz} obtain a weaker upper bound in the setting of Theorem \ref{thm-main} by using certain manipulations (relying on combinatorial arguments) to reduce to the setting of diagonal tensors for applying the slice rank polynomial method (or variants thereof).

Here, we obtain a much stronger upper bound with a new approach that incorporates both (a consequence of) the slice rank polynomial method and a higher-uniformity version of the Balog--Szemer\'{e}di--Gowers Theorem. These different tools from additive combinatorics have not previously been combined, and we believe that there may be future potential in pursuing such an approach further.

In particular, our approach manages to break the ``multi-colored barrier'' for this problem. Indeed, as discussed in \cite{lisa-egz}, for the ``multi-colored'' version of the setting in Theorem \ref{thm-main} the bound $C_p\cdot (2\sqrt{p})^n$ due to the first author is essentially tight (there is a lower bound of $\sqrt{p}^n$ for even $n$). Thus, an improvement of the bound beyond $\sqrt{p}^n$ needs to use the ``single-set'' setting in Theorem \ref{thm-main} in an essential way (with an argument that is not generalizable to the ``multi-colored'' setting). So far, most arguments relying on the slice rank polynomial method naturally generalized to the ``multi-colored'' setting, and so finding new approaches that are specific to the ``single-set'' setting has been a major challenge. In particular, this challenge also appears for the problem of improving the bounds of Ellenberg--Gijswijt \cite{ellenberg-gijswijt} for the cap-set problem and more generally the problem of bounding the size of three-term progression-free subsets of $\F_p^n$. These bounds are also (essentially) tight in the ``multi-colored setting'' (see \cite{kleinberg-sawin-speyer}), and so an approach specific to the ``single-set'' setting would be needed to overcome this ``multi-colored barrier''. Our approach of blending the slice rank polynomial method with the Balog--Szemer\'{e}di--Gowers Theorem is indeed specific ``single-set setting'', and so we believe that it may be helpful in breaking the ``multi-colored barrier'' in other problems as well.

The particular higher-uniformity version of the Balog--Szemer\'{e}di--Gowers Theorem \cite{balog-szemeredi, gowers} that we are using is due to Borenstein and Croot \cite{borenstein-croot} and is stated in Section \ref{sec-induction-step} (there are also other higher-uniformity versions, see in particular \cite{sudakov-szemeredi-vu}). Besides this theorem and a (consequence of) the slice rank polynomial method (see Section \ref{sec-slice-rank}), our proof also uses combinatorial and probabilistic arguments.

The best known lower bounds for $\s(\mathbb{Z}_m^n)$ for large $n$ are of the form $(m-1)\cdot c^n$ for all odd $m\ge 3$ with $c\approx 2.1398$ and are due to Edel \cite{edel} (in the case of $m=3$, the constant is slightly better, namely $c\approx 2.2180$ due to Tyrell \cite{tyrrell}). This in particular leads to the lower bound $\s(\mathbb{Z}_m^n)\ge 2.1398^n$ for any $m$ which is not a power of $2$ and large $n$  (if $m$ is a power of $2$, then $\s(\mathbb{Z}_{m}^n)=(m-1)\cdot 2^n+1$ is known exactly). Here, the exponential base $2.1398$ is an absolute constant independent of $m$. It is still an open question whether this is the right behaviour, or whether the correct exponential base should depend on $m$:

\begin{question}
Is there an absolute constant $c$, such that for every fixed integer $m\ge 2$ we have $\s(\mathbb{Z}_m^n)\le D_m\cdot c^n$ for all $n$, where $D_m$ is a constant only depending on $m$?
\end{question}

Our upper bound in Theorem \ref{thm-main} does not give such a constant $c$, but instead it gives a term of the form $m^{o(1)}$ (where the exponent $o(1)$ converges to zero for growing $m$). In the opposite regime, where the dimension $n$ is fixed, the second author \cite{dima-egz} showed the bound $\s(\mathbb{Z}_m^n)\le m\cdot 4^n$ if all prime factors of $m$ are sufficiently large with respect to $n$.

\textit{Acknowledgements.} The authors would like to thank Cosmin Pohoata for helpful conversations and for pointing out reference \cite{borenstein-croot}, as well as Jacob Fox for useful comments on an earlier version of this paper.

\textit{Notation.} For a subset $A$ of an additively written abelian group (for us, the group will usually be $\Fpn$), we write $\ell A=A+\dots+A=\{x_1+\dots+x_\ell\mid x_1,\dots,x_\ell\in A\}$, as usual. The cover number of a finite family $\mathcal{F}$ of non-empty subsets $X\su S$ of some ground set $S$ is the size of the smallest subset $S'\su S$ such that $S'$ intersects every set $X\in \mathcal{F}$.

\section{Proof Overview}

\subsection{Proof Structure}

The Balog--Szemer\'{e}di--Gowers Theorem gives, for a subset $A\su \Fpn$ with many solutions to the equation $y_1+y_2=y_1'+y_2'$ with $y_1,y_2,y_1',y_2'\in A$, a subset $A'\su A$ such that the sum-set $A'+A'$ is small. Similarly, the higher uniformity version of the theorem due to Borenstein--Croot \cite{borenstein-croot} gives, under suitable conditions on $A$, a subset $A'\su A$ such that the $\ell$-fold sum-set $\ell A'=A'+\dots+A'$ is small for certain $\ell$. A priori, it is unclear how such a subset $A'$ is useful for finding a solution to $x_1+\dots+x_p=0$ with distinct vectors $x_1,\dots,x_p\in A$.

The key for taking advantage of such a statement lies in the inductive setup for our proof, which allows us to incorporate the higher uniformity Balog--Szemer\'{e}di--Gowers Theorem in interplay with the slice rank polynomial method. The actual statement that we induct on is as follows.

\begin{theorem}\label{thm-technical}
For every $c> 1$, there is a positive integer $\ell$ and a constant $C(c)\ge 1$ such that for every sufficiently large prime $p$ (large enough in terms of $c$) there is a constant $D(c,p)\ge 1$ such that the following holds for all $n$. If $A\su \Fpn$ is a subset of size $|A|\ge D(c,p)\cdot (C(c))^n$ with $|\ell A|=|A+\dots+A|\le |A|^c$, then $A$ contains $p$ distinct vectors $x_1,\dots,x_p\in A$ with $x_1+\dots+x_p=0$.
\end{theorem}

We prove this statement inductively (taking $c=h/2$ and inducting on $h=3,4,5,\dots$). Assuming Theorem \ref{thm-technical}, it is not difficult to deduce Theorem \ref{thm-main}:

\begin{proof}[Proof of Theorem \ref{thm-main} assuming Theorem \ref{thm-technical}]
As in Theorem \ref{thm-main}, let $\eps>0$ be fixed. Noting that the statement in Theorem \ref{thm-main} is trivial for $\eps\ge 1$ (since we always have $|A|\le p^n$), we may assume that $0<\eps<1$. Now, let $c=1/\eps$, and let $\ell$ and $C(c)$ be as in Theorem \ref{thm-technical}. Let us choose $C_\eps>C(c)$ large enough such that the statement in Theorem \ref{thm-technical} holds for all primes $p\ge C_\eps$.

As in Theorem \ref{thm-main}, let us now consider a prime $p$. Note that for any prime $p< C_\eps$ and any subset $A\su \Fpn$ we trivially have $|A|\le p^n\le (C_\eps p^{\eps})^n$. Thus, we may assume that $p\ge  C_\eps$, and so there is a constant $D(c,p)$ such that the statement in Theorem \ref{thm-technical} holds. Let $D_{\eps,p}=D(c,p)$.

Suppose that $A\su \Fpn$ is a subset of size $|A|> D_{\eps,p}\cdot (C_\eps p^{\eps})^n$  not containing distinct vectors $x_1,\dots,x_p\in A$ with $x_1+\dots+x_p=0$. Note that we have $|A|> D_{\eps,p}\cdot (C_\eps p^{\eps})^n\ge D(c,p)\cdot (C(c))^n$ and
\[|\ell A|=|A+\dots+A|\le p^n= (p^{\eps n})^c\le |A|^c.\]
Thus, we obtain a contradiction to Theorem \ref{thm-technical}.
\end{proof}

It is also not difficult to show that Theorem \ref{thm-main} implies our main result about Erd\H{o}s--Ginzburg--Ziv constants in Theorem \ref{thm-egz-integer-m}.

\begin{proof}[Proof of Theorem \ref{thm-egz-integer-m} assuming Theorem \ref{thm-main}]
As in the statement of Theorem \ref{thm-egz-integer-m}, let $\eps>0$ and $m\ge 2$ be fixed. We have
\[\s(\mathbb{Z}_m^n)<m\cdot \sum_{p}\frac{\s(\mathbb{Z}_p^n)}{p-1},\]
where the sum is over all prime factors $p$ of $m$ (see, for example, \cite[Lemma 11]{fox-s-egz}). On the other hand, for each prime factor $p$ of $m$, we can bound $\s(\mathbb{Z}_p^n)$ using Theorem \ref{thm-main} as follows. Consider a sequence of elements of $\mathbb{Z}_p^n\cong \Fpn$ without a subsequence of length $p$ summing to zero. Clearly, this sequence can contain at most $p-1$ copies of any particular vector in $\Fpn$. On the other hand, the set $A\su \Fpn$ of all vectors appearing at least once in the sequence does not contain $p$ distinct vectors summing to zero, and so by Theorem \ref{thm-main} we have $|A|\le D_{\eps,p}\cdot (C_\eps p^\eps)^n$. This means that the sequence has length at most $(p-1)\cdot D_{\eps,p}\cdot (C_\eps p^\eps)^n$ and hence $\s(\mathbb{Z}_p^n)\le (p-1)\cdot D_{\eps,p}\cdot (C_\eps p^\eps)^n+1\le (p-1)\cdot (D_{\eps,p}+1)\cdot (C_\eps p^\eps)^n$. Thus, we obtain
\[\s(\mathbb{Z}_m^n)<m\cdot \sum_{p}(D_{\eps,p}+1)\cdot (C_\eps p^\eps)^n,\]
where the sum is again over all prime factors $p$ of $m$. Thus, the desired statement holds when taking the constant $D_{\eps,m}$ in Theorem \ref{thm-egz-integer-m} to be the sum of $m\cdot (D_{\eps,p}+1)$ for the constants $D_{\eps,p}$ in Theorem \ref{thm-main} over all prime factors $p$ of $m$ (and taking $C_\eps$ to be the same constant as in Theorem \ref{thm-main}).
\end{proof}

The main difficulty is of course to prove Theorem \ref{thm-technical}, and the rest of this paper is devoted to this. We start by giving an outline of the main ideas of the proof in the next subsection.

\subsection{Outline of proof of Theorem \ref{thm-technical}}

Noting that Theorem \ref{thm-technical} gets strictly stronger as we increase $c$, we may assume that $c=h/2$ for an integer $h\ge 3$. We will then prove Theorem \ref{thm-technical} by induction on $h$.

To prove the theorem, we need to show that there is a solution to $x_1+\dots+x_p=0$ with $p$ distinct vectors $x_1,\dots,x_p\in A$. Our strategy for finding such a solution is to start with a solution to $x_1+\dots+x_p=0$ where $x_1,\dots,x_p\in A$ are not necessarily distinct, and then modify this solution to make $x_1,\dots,x_p$ distinct. More specifically, to modify the solution we split the vectors $x_1,\dots,x_p$ into $\ell$-tuples (and a small remainder of fewer than $\ell$ vectors), and then we replace each $\ell$-tuple by another $\ell$-tuple of vectors in $A$ with the same sum. Since at every step we replace $\ell$ vectors in $A$ with $\ell$ different vectors in $A$ with the same sum, the sum $x_1+\dots+x_p$ does not change throughout this process, and so at every step our vectors $x_1,\dots,x_p\in A$ form a solution to $x_1+\dots+x_p=0$. The difficulty is, however, to obtain a solution where $x_1,\dots,x_p$ are distinct.

It turns out that using the slice rank polynomial method and some combinatorial arguments, we can ensure that at the start of our process we have a solution to $x_1+\dots+x_p=0$ with $x_1,\dots,x_p\in A$ that can be split into $\ell$-tuples (and fewer than $\ell$ remaining vectors) in such a way that each $\ell$-tuple consists of $\ell$ distinct vectors (and also such that the vectors in the remainder are distinct from each other). Our aim in each step of the process is now to replace one of the $\ell$-tuples with a different $\ell$-tuple of distinct vectors in $A$ with the same sum, such that the $\ell$ vectors in the new $\ell$-tuple are distinct from all the other vectors appearing among our solution to $x_1+\dots+x_p=0$ at that step. If we can do this step by step for each $\ell$-tuple, this greedy procedure will lead to a solution to $x_1+\dots+x_p=0$ with distinct vectors $x_1,\dots,x_p\in A$.

Of course, it may happen that at some step, when we want to replace a certain $\ell$-tuple with sum $w\in \Fpn$, we cannot find an $\ell$-tuple of distinct vectors in $A$ with the same sum $w$ which is disjoint from all vectors currently appearing in our solution to $x_1+\dots+x_p=0$. In this case, every subset $\{y_1,\dots,y_\ell\}\su A$ consisting of distinct vectors $y_1,\dots,y_\ell\in A$ with sum $y_1+\dots+y_\ell=w$ must contain one of the vectors $x_1,\dots,x_p$. Hence the family of subsets $\{y_1,\dots,y_\ell\}\su A$ with distinct elements $y_1,\dots,y_\ell\in A$ with sum $y_1+\dots+y_\ell=w$ must have a cover of size at most $p$. We call a vector $w\in \Fpn$ \emph{bad} if this happens. Furthermore, we call an $\ell$-tuple of $\ell$ distinct vectors in $A$ \emph{bad}, if the sum of the $\ell$ vectors is bad. Then at every step of our process, if the relevant $\ell$-tuple is not bad, we will be able to replace it in the desired way.

Thus, if at the start of our process each of the $\ell$-tuples into which we split our starting solution to $x_1+\dots+x_p=0$ is not a bad $\ell$-tuple, we will be able to run this modification process for each of the $\ell$-tuples and obtain a solution to $x_1+\dots+x_p=0$ with distinct vectors $x_1,\dots,x_p\in A$ in the end. So it suffices to find a solution to $x_1+\dots+x_p=0$ with $x_1,\dots,x_p\in A$ that can be spit into $\ell$-tuples in such a way, that each of these $\ell$-tuples consists of $\ell$ distinct vectors and is not bad (and such that the fewer than $\ell$ remaining vectors are distinct from each other).

If there are only few bad $\ell$-tuples $(y_1,\dots,y_\ell)\in A^\ell$, then by a probabilistic subset sampling argument there is a relatively large subset $A'\su A$ such that there exists no bad $\ell$-tuple $(y_1,\dots,y_\ell)\in A^\ell$ with $y_1,\dots,y_\ell\in A'$. Then, relying on the slice rank polynomial method and further combinatorial arguments, one can find a solution to $x_1+\dots+x_p=0$ with $x_1,\dots,x_p\in A'$ that can be split into $\ell$-tuples in the desired way (and then automatically none of these $\ell$-tuples will be bad). Applying our process as discussed above, we can turn this into a solution to $x_1+\dots+x_p=0$ with $x_1,\dots,x_p\in A$ such that $x_1,\dots,x_p$ are distinct.

For the induction beginning, i.e.\ the case $c=3/2$ in Theorem \ref{thm-technical}, this already suffices. In fact, if $c=3/2$, we can take $\ell=2$ in Theorem \ref{thm-technical} and observe that then by the assumption $|A+A|\le |A|^{3/2}$ the number of bad $2$-tuples is at most $2p\cdot |A|^{3/2}$ (indeed, every possible sum $w=x_1+x_2$ can lead to at most $2p$ bad $2$-tuples $(x_1,x_2)\in A^2$ with $x_1+x_2=w$). Thus, there are only few bad $2$-tuples in $A^2$, and the above argument applies.

In the induction step, we also need to consider a second case, namely that there are many bad $\ell$-tuples $(x_1,\dots,x_\ell)\in A^\ell$. For each bad $\ell$-tuple $(x_1,\dots,x_\ell)$, the sum $w=x_1+\dots+x_\ell$ is bad, and one of the vectors $x_1,\dots,x_\ell$ is in the cover of size at most $p$ of the family of all $\ell$-element subsets of $A$ with sum $w$. Thus, upon reordering the vectors, every bad $\ell$-tuple can be rewritten as $(y_1,\dots,y_{\ell-1},z)$ where $z$ is in the cover of size at most $p$ of the family of all $\ell$-element subsets of $A$ with sum $y_1+\dots+y_{\ell-1}+z$. By our assumption $|\ell A|\le |A|^c$ in Theorem \ref{thm-technical}, there are not so many possibilities for the sum $y_1+\dots+y_\ell+z$, and for each of these possibilities there are at most $p$ choices for $z$. One can now show that there must be many bad $\ell$-tuples $(y_1,\dots,y_{\ell-1},z)\in A^\ell$, for which the sums $y_1+\dots+y_{\ell-1}$ are concentrated on relatively few possible values. This means that one can apply the higher uniformity Balog--Szemer\'{e}di--Gowers Theorem due to Borenstein--Croot \cite{borenstein-croot} (whose precise statement is given in Theorem \ref{thm-borenstein-croot}). This theorem (with suitably chosen parameters) implies that there is a relatively large subset $A'\su A$ with $|\ell'A'|\le |A|^{c-1/2}$ for the value $\ell'$ such that Theorem \ref{thm-technical} holds for $c':=c-1/2$ with $\ell'$ instead of $\ell$ (recall that Theorem \ref{thm-technical} holds for $c-1/2=(h-1)/2$ by our induction hypothesis). Now, by the conclusion of Theorem \ref{thm-technical} for $c'$ and and $\ell'$, we can find a solution to $x_1+\dots+x_p=0$ with distinct vectors $x_1,\dots,x_p\in A'\su A$.

This finishes the outline of our proof of Theorem \ref{thm-technical}. The actual proof can be found in Section \ref{sect-proof-technical-thm}, after some preparations for the proof in Section \ref{sect-preparations}.

\section{Preparations}
\label{sect-preparations}

\subsection{Solutions with not too many repetitions}
\label{sec-slice-rank}

In this subsection, we show that any large subset $A\su \Fpn$ must contain a solution to the equation $x_1+\dots+x_p=0$ with $x_1,\dots,x_p\in A$, such that no vector appears a lot of times among $x_1,\dots,x_p$. The precise statement is given in Proposition \ref{prop-repetitions-solutions} below. This will help us to find a solution to $x_1+\dots+x_p=0$ with $x_1,\dots,x_p\in A$, such that $(x_1,\dots,x_p)$  can be split into $\ell$-tuples (and fewer than $\ell$ remaining vectors) such that each $\ell$-tuple consists of $\ell$ distinct vectors.

The proof of our proposition uses the $k$-coloured Sum-Free Theorem, which can be proved with the slice rank polynomial method. As mentioned in the introduction, the slice rank polynomial method was introduced by Tao \cite{tao} following work of Croot--Lev--Pach \cite{croot-lev-pach} and Ellenberg--Gijswijt \cite{ellenberg-gijswijt}. A proof of the $k$-coloured Sum-Free Theorem following this method can be found in \cite{lovasz-sauermann}.

\begin{theorem}[$k$-coloured Sum-Free Theorem]\label{thm-mulitcolor-sum-free}
Let $k\ge 3$ be an integer, and let $p$ be a prime. For some positive integer $n$, let $(x_1^{(j)},\dots,x_k^{(j)})\in \Fpn\times \dots\times \Fpn$ for $j=1,\dots,L$ be a list of $k$-tuples of vectors in $\Fpn$. Suppose that for all $j_1,\dots,j_k\in \{1,\dots,L\}$, we have
\[x_1^{(j_1)}+\dots+x_k^{(j_k)}=0\quad\text{if and only if}\quad j_1=\dots=j_k.\]
Then we must have $L\le (\Gamma_{p,k})^n$, where
\[\Gamma_{p,k}=\inf_{0<\gamma<1}\frac{1+\gamma+\dots+\gamma^{p-1}}{\gamma^{(p-1)/k}}<p.\]
\end{theorem}

As an immediate consequence of the $k$-coloured Sum-Free Theorem one can show that every subset $A\su \Fpn$ of size $|A|\ge 4^n>(\Gamma_{p,p})^n$ must contain a solution to the equation $y_1+\dots+y_p=0$ such that $y_1,\dots,y_p\in A$ are not all equal. In other words, we obtain a solution such that every vector in $\Fpn$ appears at most $p-1$ times among $y_1,\dots,y_p$.

The statement of the following proposition is somewhat similar, showing that every large enough subset  $A\su \Fpn$ contains a solution to $y_1+\dots+y_p=0$, such that every vector in $\Fpn$ appears at most $\lambda p$ times among $y_1,\dots,y_p$ (for some fixed $0<\lambda\le 1$).

\begin{proposition}\label{prop-repetitions-solutions}
For every fixed $0<\lambda\le 1$, there exists a constant $C'_\lambda\ge 1$ such that for every prime $p> 1/\lambda$ and every positive integer $n$ the following holds. If $A\su \Fpn$ is a subset of size $|A|> p^2\cdot (C'_\lambda)^n$, there exist vectors $y_1,\dots,y_p\in A$ with $y_1,+\dots+y_p=0$ such that every vector in $\Fpn$ appears among $y_1,\dots,y_p$ at most $\lambda p$ times.
\end{proposition}

\begin{proof}
We define
\[C'_\lambda=\inf_{0<\gamma<1}\frac{1}{(1-\gamma)\cdot \gamma^{1/\lambda}}.\]
Note that then for all primes $p>1/\lambda$ we have $\lceil \lambda p\rceil+1\ge 3$ and
\[\Gamma_{p,\lceil \lambda p\rceil+1}=\inf_{0<\gamma<1}\frac{1+\gamma+\dots+\gamma^{p-1}}{\gamma^{(p-1)/(\lceil \lambda p\rceil+1)}}\le \inf_{0<\gamma<1}\frac{1+\gamma+\dots+\gamma^{p-1}}{\gamma^{1/\lambda}}\le \inf_{0<\gamma<1}\frac{1}{(1-\gamma)\cdot \gamma^{1/\lambda}}=C'_\lambda,\]
and furthermore also
\[\Gamma_{p,p}=\inf_{0<\gamma<1}\frac{1+\gamma+\dots+\gamma^{p-1}}{\gamma^{(p-1)/p}}\le \inf_{0<\gamma<1}\frac{1}{(1-\gamma)\cdot \gamma}\le \inf_{0<\gamma<1}\frac{1}{(1-\gamma)\cdot \gamma^{1/\lambda}}=C'_\lambda,\]

Now, as in the statement of the proposition, let $A\su \Fpn$ be a subset of size $|A|\ge p^2\cdot (C'_\lambda)^n$. Let us suppose for contradiction that for any solution to the equation $y_1+\dots+y_p=0$ with $y_1,\dots,y_p\in A$, there is a vector appearing among $y_1,\dots,y_p$ at least $\lceil \lambda p\rceil$ times.

For every solution $(y_1,\dots,y_p)\in A^p$ to the equation $y_1+\dots+y_p=0$ let us consider the number $|\{y_1,\dots,y_p\}|$, i.e.\ the number of distinct vectors appearing among $y_1,\dots,y_p$ (this number is in $\{1,\dots,p\}$). Let us say that two solutions $(y_1,\dots,y_p),(y'_1,\dots,y'_p)\in A^p$ to this equation are disjoint if no vector appears in both of them.

\begin{claim}\label{claim-set-A-prime}
There exists a number $\ell\in \{1,\dots,p\}$ and a subset $A'\su A$ satisfying the following two conditions:
\begin{itemize}
\item[(i)] There is a collection of more than $(C'_\lambda)^n$ pairwise disjoint solutions $(y_1,\dots,y_p)\in (A')^p$ to the equation $y_1+\dots+y_p=0$ with $|\{y_1,\dots,y_p\}|=\ell$.
\item[(ii)] Every solution $(y_1,\dots,y_p)\in (A')^p$ to the equation $y_1+\dots+y_p=0$ satisfies $|\{y_1,\dots,y_p\}|\le \ell$.
\end{itemize}
\end{claim}
\begin{proof}
Let us  define a sequence of subsets $A'_p\supseteq A'_{p-1}\supseteq\dots \supseteq A'_\ell$ of $A\su \Fpn$ for some $\ell\in \{0,\dots,p\}$ with the following recursive process. Throughout this process we will ensure that for every $j=\ell,\dots, p$, every solution $(y_1,\dots,y_p)\in (A'_j)^p$ to the equation $y_1+\dots+y_p=0$ satisfies $|\{y_1,\dots,y_p\}|\le j$.

We start by defining $A'_p=A'$. Clearly every solution $(y_1,\dots,y_p)\in (A'_p)^p$ to the equation $y_1+\dots+y_p=0$ satisfies $|\{y_1,\dots,y_p\}|\le p$.

Suppose that for some index $1\le j\le p$, we have already defined the set $A'_j\su A\su \Fpn$ with the property that every solution $(y_1,\dots,y_p)\in (A'_j)^p$ to the equation $y_1+\dots+y_p=0$ satisfies $|\{y_1,\dots,y_p\}|\le j$.

Let us now consider a maximal collection of pairwise disjoint solutions $(y_1,\dots,y_p)\in (A'_j)^p$ to the equation $y_1+\dots+y_p=0$ with $|\{y_1,\dots,y_p\}|=j$. If this maximal collection has size larger than $(C'_\lambda)^n$, then let us terminate the process and define $\ell=j$ (then we do not need to define another set $A'_{j-1}$). Otherwise, this maximal collection has size at most $(C'_\lambda)^n$ and so there are at most $p\cdot (C'_\lambda)^n$ different vectors appearing in one of the solutions $(y_1,\dots,y_p)$ in our collection in $(A'_j)^p$. Now, let the set $A'_{j-1}$ be obtained from $A'_j$ by deleting all the vectors appearing in some solution in the collection. Note that then, by the maximality of the chosen collection, no solutions $(y_1,\dots,y_p)$ to $y_1+\dots+y_p=0$ with $|\{y_1,\dots,y_p\}|=j$ remain. Hence in the set $A'_{j-1}$ every solution $(y_1,\dots,y_p)\in (A'_{j_1})^p$ to the equation $y_1+\dots+y_p=0$ satisfies $|\{y_1,\dots,y_p\}|\le j-1$.

This process defines subsets $A'_p\supseteq A'_{p-1}\supseteq\dots \supseteq A'_\ell$ of $A\su \Fpn$ for some $\ell\in \{0,\dots,p\}$. Note that at every step of the process we delete at most $p\cdot (C'_\lambda)^n$ vectors, meaning that $|A'_{j-1}|\ge |A'_j|-p\cdot (C'_\lambda)^n$  for $\ell+1\le j\le p$. This implies that
\[|A'_\ell|\ge  |A'_p|-(\ell-p)\cdot p\cdot (C'_\lambda)^n\ge |A|-p^2\cdot (C'_\lambda)^n>0,\]
so the final set $A'_\ell$ is non-empty.

We claim that $\ell\ne 0$. Indeed, if we had $\ell=0$, then $A'_0$ would be a non-empty subset of $\Fpn$ such that every solution $(y_1,\dots,y_p)\in (A'_{j_1})^p$ to the equation $y_1+\dots+y_p=0$ satisfies $|\{y_1,\dots,y_p\}|\le 0$. This is a contradiction, since for any $y\in A'_0$ we can form a solution solution $(y_1,\dots,y_p)\in (A'_0)^p$ to $y_1+\dots+y_p=0$ by taking $(y_1,\dots,y_p)=(y,\dots,y)$ and then we have $|\{y_1,\dots,y_p\}|=1$. Thus, we must have $\ell\in \{1,\dots,p\}$.

This means that the process above terminated with the set $A'_\ell$, which means that $A'_\ell$ contains a collection of more than $(C'_\lambda)^n$ pairwise disjoint solutions $(y_1,\dots,y_p)\in (A'_\ell)^p$ to the equation $y_1+\dots+y_p=0$ with $|\{y_1,\dots,y_p\}|=\ell$. Thus, taking $A'=A'_\ell$, condition (i) is satisfied. Furthermore, condition (ii) is satisfied, since throughout the process we maintained the property that every solution $(y_1,\dots,y_p)\in (A'_\ell)^p$ to the equation $y_1+\dots+y_p=0$ satisfies $|\{y_1,\dots,y_p\}|\le \ell$.
\end{proof}

As in Claim \ref{claim-set-A-prime}, let us choose $\ell\in \{1,\dots,p\}$ and $A'\su A$ satisfying conditions (i) and (ii). By condition (i), there exists a collection $\mathcal{C}\su (A')^p$ of $|\mathcal{C}|>(C'_\lambda)^n$ pairwise disjoint solutions $(y_1,\dots,y_p)\in (A')^p$ to the equation $y_1+\dots+y_p=0$ with $|\{y_1,\dots,y_p\}|=\ell$. By our assumption, for each of these solutions there is a vector appearing among $y_1,\dots,y_p$ at least $\lceil \lambda p\rceil$ times. So let us define $k=\lceil \lambda p\rceil+1$ (then $k-1=\lceil \lambda p\rceil$) and let us re-order the vectors in each solution $(y_1,\dots,y_p)\in \mathcal{C}$ in our collection in such a way that $y_1=\dots=y_{k-1}$.

Let $L=|\mathcal{C}|> (C'_\lambda)^n$, and let $(y_1^{(j)},\dots,y_p^{(j)})$ for $j=1,\dots,L$ be the $p$-tuples in $\mathcal{C}\su (A')^p$. Then for every $j=1,\dots,L$ we have $y_1^{(j)}+\dots+y_p^{(j)}=0$ and $|\{y_1^{(j)},\dots,y_p^{(j)}\}|=\ell$ as well as $y^{(j)}_1=\dots=y^{(j)}_{k-1}$. This implies that $|\{y_k^{(j)},\dots,y_p^{(j)}\}\setminus \{y_1^{(j)}\}|\ge \ell-1$. Furthermore, for distinct $j,j'\in \{1,\dots,L\}$ the $p$-tuples $(y_1^{(j)},\dots,y_p^{(j)})$ and $(y_1^{(j')},\dots,y_p^{(j')})$ are disjoint, and so we can conclude that
\begin{equation}\label{eq-distinct-variables-slice-rank-argument}
|\{y_k^{(j)},\dots,y_p^{(j)}\}\setminus \{y_1^{(1)},\dots,y_L^{(1)}\}|\ge \ell-1
\end{equation}
for $j=1,\dots,L$.

Suppose we have $\lceil \lambda p\rceil=p$. Then by our assumption for any solution to the equation $y_1+\dots+y_p=0$ with $y_1,\dots,y_p\in A$, there is a vector appearing among $y_1,\dots,y_p$ at least $\lceil \lambda p\rceil=p$ times. In other words,  for any solution to the equation $y_1+\dots+y_p=0$ with $y_1,\dots,y_p\in A$ we must have $y_1=\dots=y_p$. This implies that there cannot be any solution of the form $(y_1^{(j_1)},\dots,y_p^{(j_p)})$ with $y_1^{(j_1)}+\dots +y_p^{(j_p)}=0$ where $j_1,\dots,j_p\in \{1,\dots,L\}$ are not all equal (indeed, if $j_i\neq j_{i'}$, then $y_i^{(j_i)}\ne y_{i'}^{(j_{i'})}$ as $(y_1^{(j_i)},\dots,y_p^{(j_i)})$ and $(y_1^{(j_{i'})},\dots,y_p^{(j_{i'})})$ disjoint). Hence the $p$-tuples $(y_1^{(j)},\dots,y_p^{(j)})$ for $j=1,\dots,L$ satisfy the assumptions of Theorem \ref{thm-mulitcolor-sum-free} for $k=p$. On the other hand, we have $L>(C'_\lambda)^n\ge (\Gamma_{p,p})^n$ which is a contradiction to the conclusion of Theorem \ref{thm-mulitcolor-sum-free}.

So we may now assume that $2\le \lceil \lambda p\rceil\le p-1$, meaning that  $2\le k-1\le p-1$. For $j=1,\dots,L$, let us define a $k$-tuple $(x_1^{(j)},\dots,x_k^{(j)})\in \Fpn\times\dots\times \Fpn$ by setting $x_i^{(j)}=y_i^{(j)}=y_1^{(j)}$ for $i=1,\dots,k-1$ and
\[x_k^{(j)}=y_k^{(j)}+\dots+y_p^{(j)}=-(y_1^{(j)}+\dots+y_{k-1}^{(j)})=-(k-1)\cdot y_1^{(j)}.\]
Note that then we have $x_1^{(j)}+\dots+x_k^{(j)}=y_1^{(j)}+\dots+y_{k-1}^{(j)}+y_k^{(j)}+\dots+y_p^{(j)}=0$ for $j=1,\dots,L$. Since $L>(C'_\lambda)^n\ge (\Gamma_{p,k})^n$, from Theorem \ref{thm-mulitcolor-sum-free} we can conclude that there must exist $j_1,\dots j_k\in \{1,\dots,L\}$ with $x_1^{(j_1)}+\dots +x_k^{(j_k)}=0$ such that $j_1,\dots,j_k$ are not all equal.

Suppose we have $j_1=\dots=j_{k-1}$, then let $j=j_1=\dots=j_{k-1}$ and observe that $j_k\ne j$ (since $j_1,\dots,j_k$ are not all equal). But now we have
\[0=x_1^{(j_1)}+\dots +x_{k-1}^{(j_{k-1})}+x_k^{(j_k)}=x_1^{(j)}+\dots +x_{k-1}^{(j)}-(k-1)\cdot y_1^{(j_k)}=(k-1)\cdot y_1^{(j)}-(k-1)\cdot y_1^{(j_k)}\]
and since $k-1\ne 0$ in $\mathbb{F}_p$ this implies that $y_1^{(j)}=y_1^{(j_k)}$. But this is a contradiction since the $p$-tuples $(y_1^{(j)},\dots,y_p^{(j)})$ and $(y_1^{(j_k)},\dots,y_p^{(j_k)})$ are disjoint.

So let us now assume that the indices $j_1,\dots,j_{k-1}$ are not all equal. Then the set
\[\{x_1^{(j_1)},\dots ,x_{k-1}^{(j_{k-1})}\}=\{y_1^{(j_1)},\dots ,y_{k-1}^{(j_{k-1})}\}=\{y_1^{(j_1)},\dots ,y_1^{(j_{k-1})}\}\]
has size at least $2$ (since the $p$-tuples $(x_1^{(j)},\dots,x_p^{(j)})$ for $j=1,\dots,L$ are pairwise disjoint). Now, we have
\[0=x_1^{(j_1)}+\dots +x_{k-1}^{(j_{k-1})}+x_k^{(j_k)}=y_1^{(j_1)}+\dots +y_{k-1}^{(j_{k-1})}+y_k^{(j)}+\dots+y_p^{(j)}=y_1^{(j_1)}+\dots +y_{1}^{(j_{k-1})}+y_k^{(j)}+\dots+y_p^{(j)},\]
So $(y_1^{(j_1)},\dots ,y_{1}^{(j_{k-1})},y_k^{(j)},\dots,y_p^{(j)})\in (A')^p$ satisfies $y_1^{(j_1)}+\dots +y_{1}^{(j_{k-1})}+y_k^{(j)}+\dots+y_p^{(j)}=0$ and
\[|\{y_1^{(j_1)},\dots ,y_{1}^{(j_{k-1})},y_k^{(j)},\dots,y_p^{(j)}\}|\ge |\{y_1^{(j_1)},\dots ,y_{1}^{(j_{k-1})}\}|+ |\{y_k^{(j)},\dots,y_p^{(j)}\}\setminus \{y_1^{(1)},\dots,y_L^{(1)}\}|\ge 2+(\ell-1)=\ell+1,\]
where for the second inequality we used (\ref{eq-distinct-variables-slice-rank-argument}). But this is a contradiction to condition (ii) in our choice of $\ell$ and $A'$ as in Claim \ref{claim-set-A-prime}. This finishes the proof of Proposition \ref{prop-repetitions-solutions}.\end{proof}

\subsection{Partitioning into rainbow sets}

This subsection proves the following combinatorial lemma about partitioning a set with coloured elements into rainbow subsets. This, together with the results from the last subsection, allows us to split our solution $(x_1,\dots,x_p)$ of $x_1+\dots+x_p$ into $\ell$-tuples, each consisting of $\ell$ distinct vectors, in the desired way.

\begin{lemma}\label{lem-rainbow-sets}
Let $1\le \ell\le k$ be integers, and consider a colouring of a set $S$ of size $|S|=k$ (assigning each of the elements in $S$ a colour). Suppose that each colour occurs at most $k/\ell$ times. Then there is a partition $S=S_1\cup \dots\cup S_{\lfloor k/\ell\rfloor}\cup S_{\lfloor k/\ell\rfloor+1}$ with $|S_j|=\ell$ for $j=1,\dots,\lfloor k/\ell\rfloor$ and $|S_{\lfloor k/\ell\rfloor+1}|=k-\lfloor k/\ell\rfloor\cdot \ell$ such that for each $j=1,\dots,\lfloor k/\ell\rfloor+1$ all elements of $S_j$ have distinct colours.
\end{lemma}

\begin{proof}
    Let $m=\lfloor k/\ell\rfloor$, and $r=k-m\ell\in \{0,\dots,\ell-1\}$. Let us label the elements of the set $S$ as $s_1, \ldots, s_k$ in such a way that each colour class forms a block of consecutive elements. It now suffices to show that we can find a partition $\{1,\dots,k\}=I_1\cup \dots\cup I_{m}\cup I_{m+1}$ with $|I_j|=\ell$ for $j=1,\dots,m$ and $|S_{m+1}|=r$ such that for each $j=1,\dots,m+1$ we have $|x-y|\ge m$ for any two distinct elements $x,y\in S_j$. Indeed, then we can define $S_j=\{s_x\mid x\in I_j\}$ for $j=1,\dots,m+1$. Note that then for any $j=1,\dots,m+1$ and any distinct $x,y\in I_j$, the elements $s_x$ and $s_y$ cannot have the same colour, since otherwise all elements between $s_x$ and $s_y$ in the list $s_1, \ldots, s_k$ would also be of that colour, and so the colour would appear at least $m+1>k/\ell$ times since $|x-y|\ge m$.

    To define the desired partition $\{1,\dots,k\}=\{1,\dots,m\ell+r\}=I_1\cup \dots\cup I_{m}\cup I_{m+1}$, let $I_{m+1}=\{m+1,2m+2,\dots,rm+r\}$ and $I_j=\{x\in \{1,\dots,k\}\mid x\equiv j \pmod{m}\}\setminus I_{m+1}$ for $j=1,\dots,m$. It is not hard to see that $|I_j|=\ell$ for $j=1,\dots,m$ and $|S_{m+1}|=r$. Furthermore, for any two distinct $x,y\in S_{m+1}$, we have $|x-y|\ge m+1> m$, since $x$ and $y$ are both multiples of $m+1$. For any $j=1,\dots,m$ and any two distinct $x,y\in S_{j}$, we have $|x-y|\ge m$, since $x\equiv j\equiv y \pmod{m}$.
\end{proof}

\section{Proof of Theorem \ref{thm-technical}}
\label{sect-proof-technical-thm}

In this section, we finally prove Theorem \ref{thm-technical}. Note that if the theorem holds for some $c>1$, then it also holds for all smaller values of $c$. Hence it suffices to prove the theorem for $c=h/2$ for all integers $h\ge 3$.

We prove Theorem \ref{thm-technical} for $c=h/2$ for $h=3,4,\dots$ by induction on $h$. The first subsection of this section contains the induction beginning $h=3$, and the second subsection contains the induction step.

\subsection{Induction beginning \texorpdfstring{$\boldsymbol{c=3/2}$}{c=3/2}}

As the starting point of our induction, let us prove Theorem \ref{thm-technical} for $c=3/2$. The following lemma shows that the desired statement holds for $\ell=2$ and $C(c)=(C'_{1/2})^2$ (with $C'_{1/2}$ as in Proposition \ref{prop-repetitions-solutions}) and $D(p,3/2)=9p^6$ for any prime $p\ge 3$.

\begin{lemma}
Let $C'_{1/2}\ge 1$ be as in Proposition \ref{prop-repetitions-solutions} for $\lambda=1/2$. Let $p\ge 3$ be a prime and let $n$ be a positive integer. Suppose that $A\su \Fpn$ is a subset of $\Fpn$ such that $|A|\ge 9p^6\cdot (C'_{1/2})^{2n}$ and  $|A+A|\le |A|^{3/2}$. Then $A$ contains $p$ distinct vectors $x_1,\dots,x_p\in A$ with $x_1+\dots+x_p=0$.
\end{lemma}

\begin{proof}
Let us say that a pair $\{x,y\}\su A$ with $x\ne y$ is \emph{bad} if there are at most $p$ pairs $\{x',y'\}\su A$ with $x'\ne y'$ satisfying $x'+y'=x+y$. In other words, a pair $\{x,y\}\su A$ is bad if there are at most $p$ different ways to write $x+y$ as a sum of two distinct elements of $A$. Note that every element in $A+A$ can occur as the sum of at most $p$ bad pairs (since otherwise the pairs with this sum would not be bad). Thus, there can be at most $p\cdot |A+A|\le p\cdot |A|^{3/2}$ bad pairs $\{x,y\}\su A$.

Let us consider the graph with vertex set $A$ where for distinct $x,y\in A$ we draw an edge between $x$ and $y$ if and only if $\{x,y\}$ is a bad pair. Then this graph has at most $p\cdot |A|^{3/2}$ edges, and hence has average degree at most $2p\cdot |A|^{1/2}$. Thus, by the well-known Caro--Wei bound \cite{caro,wei}, the graph has an independent set of size at least
\[\frac{|A|}{2p\cdot |A|^{1/2}+1}> \frac{|A|}{3p\cdot |A|^{1/2}}=\frac{|A|^{1/2}}{3p}\ge \frac{3p^3\cdot (C'_{1/2})^{n}}{3p}=p^2\cdot (C'_{1/2})^{n}\]
So let $A'\su A\su \Fpn$ be a subset of size $|A'|>p^2\cdot (C'_{1/2})^{n}$ such that there does not exist a bad pair $\{x,y\}\su A$ with $x,y\in A'$.

By Proposition \ref{prop-repetitions-solutions} for $\lambda=1/2$, there exist vectors $y_1,\dots,y_p\in A'$ with $y_1+\dots+y_p=0$ such that every vector in $\Fpn$ appears among $y_1,\dots,y_p$ at most $p/2$ times.

Let us now consider a colouring of the set $\{1,\dots,p\}$ where the colours correspond to the different vectors appearing among $y_1,\dots,y_p$. In other words, two indices $i,j\in\{1,\dots,p\}$ receive the same colour if and only if $y_i=y_j$. Then every colour appears at most $p/2$ times on the set $\{1,\dots,p\}$, and so by Lemma \ref{lem-rainbow-sets}, there exists a partition of $\{1,\dots,p\}$ into sets $\{i_s,j_s\}$ of size $|\{i(s),j(s)\}|=2$ for $s=1,\dots,(p-1)/2$ and one set $\{t\}$ of size $1$ such that all of the sets in this partition are rainbow.

In other words, we can split the list of vectors $y_1,\dots,y_p\in A'$ into pairs $\{y_{i(s)},y_{j(s)}\}$ with $y_{i(s)}\neq y_{j(s)}$ for $s=1,\dots,(p-1)/2$ and one remaining vector $y_t$. Recalling that $A'$ does not contain any bad pair, we observe that the pairs $\{y_{i(s)},y_{j(s)}\}\su A'$ for $s=1,\dots,(p-1)/2$ are not bad.

Let $x_t=y_t$. We will now replace each pair $\{y_{i(s)},y_{j(s)}\}\su A'$ by a pair $\{x_{i(s)},x_{j(s)}\}\su A$ with the same sum in order to construct a solution $(x_1,\dots,x_p)\in A^k$ to the equation $x_1+\dots+x_p=0$ with distinct vectors $x_1,\dots,x_p$. To do this, consider the indices $s=1,\dots,(p-1)/2$ one by one. For each index $s$, consider the sum $y_{i(s)}+y_{j(s)}\in \Fpn$. Since $\{y_{i(s)},y_{j(s)}\}$ is not bad, there are at least $p$ pairs $\{x',y'\}\su A$ with $x'\ne y'$ and $x'+y'=y_{i(s)}+y_{j(s)}$. These pairs must all be disjoint (since knowing $x'$ and the sum $x'+y'=y_{i(s)}+y_{j(s)}$ already determines $y'$), and so there must be at least one pair $\{x',y'\}\su A$ with $x'\ne y'$ and $x'+y'=y_{i(s)}+y_{j(s)}$ which does not contain $x_t$ or any of the $2s-2\le p-3$ vectors in the already chosen pairs $\{x_{i(1)},x_{j(1)}\}, \dots, \{x_{i(s-1)},x_{j(s-1)}\}$. So let us choose $\{x_{i(s)},x_{j(s)}\}\su A$ to be such a pair. Doing this step by step for $s=1,\dots,(p-1)/2$, we obtain a $p$-tuple $(x_1,\dots,x_p)\in A^p$. By construction, the vectors $x_1,\dots,x_p\in A$ are distinct and we have
\[x_1+\dots+x_p=x_t+\sum_{s=1}^{(p-1)/2} (x_{i(s-1)}+x_{j(s-1)})=y_t+\sum_{s=1}^{(p-1)/2} (y_{i(s-1)}+y_{j(s-1)})=y_1+\dots+y_p=0,\]
as desired.
\end{proof}

\subsection{Induction step}
\label{sec-induction-step}

For the induction step in the proof of Theorem \ref{thm-technical}, we will use the following result of Borenstein and Croot \cite[Theorem 4]{borenstein-croot}, which is a higher-uniformity version of the Balog--Szemer\'{e}di--Gowers Theorem \cite{balog-szemeredi,gowers}.

\begin{theorem}[\cite{borenstein-croot}]\label{thm-borenstein-croot}
For every $0<\eps<1/2$ and $c>1$, there exists $\delta>0$ such that the following holds for all sufficiently large $k$ and all sufficiently large finite subsets $A$ of an additively written abelian group. If $S\su A^k$ is a subset satisfying $|S|\ge |A|^{k-\delta}$ and $|\{y_1+\dots+y_k\mid (y_1,\dots,y_k)\in S\}|\le |A|^c$, then there is a subset $A'\su A$ of size $|A'|\ge |A'|^{1-\eps}$ such that $|\ell' A'|=|A'+\dots+A'|\le |A'|^{c(1+\eps\ell')}$ for all positive integers $\ell'$.
\end{theorem}

In order to perform the induction step for the proof of Theorem \ref{thm-technical}, let us now assume that $c=h/2$ for an integer $h\ge 4$ and that we have already proved Theorem \ref{thm-technical} for $c'=(h-1)/2=c-(1/2)$. Note that $c\ge 2$, and define $c''=c-(3/4)>1$.

Let us take a positive integer $\ell'$ as in Theorem \ref{thm-technical} for $c'=c-(1/2)$. Let us now choose $0<\eps<1/2$ small enough (depending on $c$) such that
\begin{equation}\label{eq-choice-eps}
c''\cdot (1+\eps \ell')=\left(c-\frac{3}{4}\right)\cdot(1+\eps \ell')\le c-\frac{1}{2}=c'.
\end{equation}

Let us now apply Theorem \ref{thm-borenstein-croot} to $0<\eps<1/2$ and $c''>1$. We obtain some $\delta>0$ and some positive integer $k$ such that the statement in Theorem \ref{thm-borenstein-croot} holds for all sufficiently large $A$. By decreasing $\delta$ if needed, we may assume that $0<\delta<1/4$. Let $\ell=k+1$.

Recall that we assume that Theorem \ref{thm-technical} holds for $c'$ as our induction hypothesis. So choose $C(c')\ge 1$ as well as $D(c',p)\ge 1$ for every sufficiently large prime $p$ as in Theorem \ref{thm-technical}. Furthermore, let $C'_{1/\ell}\ge 1$ be as in Proposition \ref{prop-repetitions-solutions} for $\lambda=1/\ell$. Now, let us define
\[C(c)=\max\{\ (C(c'))^2\ ,\  (C'_{1/\ell})^{\ell/\delta}\ \},\]
and for every sufficiently large prime $p$ (large enough for Theorem \ref{thm-technical} for $c'$), let us define
\[D(c,p)=\max\{\ (D(c',p))^2\ ,\  p^4\cdot p^{2\ell/\delta}\cdot 2^{(\ell+1)/\delta}\ \}.\]

Now, assuming that $p$ is sufficiently large in terms of $c$, any subset $A\su \Fpn$ of size $|A|\ge D(c,p)\cdot (C(c))^n\ge D(c,p)\ge p$ is large enough for the statement in Theorem \ref{thm-borenstein-croot}.

As in Theorem \ref{thm-technical}, let us now assume that $A\su \Fpn$ is a subset of size $|A|\ge D(c,p)\cdot (C(c))^n$ with $|\ell A|=|A+\dots+A|\le |A|^c$. We need to show that $A$ contains $p$ distinct vectors $x_1,\dots,x_p\in A$ with $x_1+\dots+x_p=0$.

Let us say that a vector $w\in \ell A\su \Fpn$ is \emph{bad} if the family of all  subsets $\{x_1,\dots,x_\ell\}\su A$ with distinct elements $x_1,\dots,x_\ell\in A$ satisfying $x_1+\dots+x_\ell=w$ has a cover $Z_w\su A$ of size $|Z_w|\le p$. Let us say that a subset $\{x_1,\dots,x_\ell\}\su A$ with distinct elements $x_1,\dots,x_\ell\in A$ is bad if the sum $x_1+\dots+x_\ell$ is bad.

Now, for every bad subset $\{x_1,\dots,x_\ell\}\su A$, one of the elements of $\{x_1,\dots,x_\ell\}$ must be in $Z_w$ for $w=x_1+\dots+x_\ell$ (indeed, $Z_w$ is a cover of all the size-$\ell$ subsets of $A$ summing to $w$). Thus, by suitably ordering, we can turn every bad subset $\{x_1,\dots,x_\ell\}$ into an $\ell$-tuple $(y_1,\dots,y_{\ell-1},z)$ such that  $w=y_1+\dots+y_{\ell-1}+z$ is bad and $z\in Z_{w}$. Hence the number of bad subsets $\{x_1,\dots,x_\ell\}\su A$ is at most the number of $\ell$-tuples $(y_1,\dots,y_{\ell-1},z)\in A^\ell$ such that  $w=y_1+\dots+y_{\ell-1}+z$ is bad and $z\in Z_{w}$.

For every $y\in \Fpn$, let us now define $N_y$ to be the number of $(\ell-1)$-tuples $(y_1,\dots,y_{\ell-1})\in A^{\ell-1}$ with $y_1+\dots+y_{\ell-1}=y$. Note that we have $N_y=0$ for all $y\not\in (\ell-1)A$.

For every $y\in \Fpn$, let us furthermore define $M_y$ to be the number of vectors $z\in A$ such that $y+z$ is bad and $z\in Z_{y+z}$. Note that we clearly have $M_y\le |A|$ for all $y\in \Fpn$.

Now, the number of $\ell$-tuples $(y_1,\dots,y_{\ell-1},z)$ such that  $w=y_1+\dots+y_{\ell-1}+z$ is bad and $z\in Z_{w}$ is precisely $\sum_{y\in \Fpn} M_yN_y$. Indeed, for every possible value of $y=y_1+\dots+y_{\ell-1}$ there are $M_y$ possibilities to choose $z\in A$ such that $w=y+z=y_1+\dots+y_{\ell-1}+z$ is bad and $z\in Z_w$, and there are furthermore $N_y$ possibilities to choose $y_1,\dots,y_{\ell-1}\in A$ with $y_1+\dots+y_{\ell-1}=y$.

Thus, the number of bad subsets $\{x_1,\dots,x_\ell\}\su A$ is at most $\sum_{y\in \Fpn} M_yN_y$.

Now, observe that
\begin{equation}\label{eq-bound-sum-N}
\sum_{y\in \Fpn} N_y= |A|^{\ell-1}.
\end{equation}
Indeed, $\sum_{y\in \Fpn} N_y$ is the total number of $(\ell-1)$-tuples $(y_1,\dots,y_{\ell-1})\in A^{\ell-1}$.

Next, we claim that $\sum_{y\in \Fpn} M_y\le p\cdot |\ell A|$. Indeed, $\sum_{y\in (\ell-1)A} M_y$ is the number of pairs $(y,z)\in \Fpn\times A$ such that $y+z$ is bad and $z\in Z_{y+z}$. We can choose such pairs by first choosing a bad $w=y+z$, then choosing $z\in Z_w$, and finally calculating $y=w-z$. Note that there are at most $|\ell A|$ choices for a bad $w$ (since by definition every bad $w$ is an element of the set $\ell A\su \Fpn$), and for every such choice of $w$ there are only $|Z_w|\le p$ choices for $z\in Z_w$. Hence the number of pairs $(y,z)\in \Fpn\times A$ such that $y+z$ is bad and $z\in Z_{y+z}$ is indeed at most $|\ell A|\cdot p$, and we indeed have $\sum_{y\in \Fpn} M_y\le p\cdot |\ell A|$.

Recalling our assumption $|\ell A|\le |A|^c$, we can now conclude that $\sum_{y\in \Fpn} M_y\le p\cdot |\ell A| \le p\cdot |A|^{c}$.
Let us now define
\[Y=\left\{y\in \Fpn \ \middle\vert\  M_y\ge \frac{|A|}{2^{\ell+1}\cdot p^{2\ell}\cdot (C'_{1/\ell})^{\ell n}}\right\}\]
Note that then we have
\[|Y|\le \frac{\sum_{y\in \Fpn} M_y}{|A|/(2^{\ell+1}\cdot p^{2\ell}\cdot (C'_{1/\ell})^{\ell n})}\le 2^{\ell+1}\cdot p^{2\ell}\cdot (C'_{1/\ell})^{\ell n}\cdot \frac{p\cdot |A|^c}{|A|}=p^{2\ell+1}\cdot (C'_{1/\ell})^{\ell n}\cdot |A|^{c-1}\le |A|^{c-(3/4)}.\]
Here, we used in the last inequality that $|A|\ge D(c,p)\cdot (C(c))^n\ge p^4\cdot p^{2\ell/\delta}\cdot (C'_{1/\ell})^{\ell n/\delta}\ge p^{8\ell+4}\cdot (C'_{1/\ell})^{4\ell n}$ (as $0<\delta<1/4$).

We will now distinguish two cases depending on the size of the sum $\sum_{y\in Y} N_y$.

\textbf{Case 1: $\boldsymbol{\sum_{y\in Y} N_y\ge |A|^{\ell-1-\delta}}$.} Recall that $k=\ell-1$ and that Theorem \ref{thm-borenstein-croot} holds with $c''=c-(3/4)$ and our chosen $0<\eps<1/2$ with our values for $\delta$ and $k$. Also recall that $A\su \Fpn$ is sufficiently large for Theorem \ref{thm-borenstein-croot} with these parameters.

Let us now define $S\su A^{k}=A^{\ell-1}$ to be the collection of $(\ell-1)$-tuples $(y_1,\dots,y_{\ell-1})\in A^{\ell-1}$ such that $y_1+\dots+y_{\ell-1}\in Y$. Then $|S|=\sum_{y\in Y} N_y\ge |A|^{\ell-1-\delta}= |A|^{k-\delta}$. We furthermore have
\[|\{y_1+\dots+y_{\ell-1}\mid (y_1,\dots,y_{\ell-1})\in S\}|\le |Y|\le |A|^{c-(3/4)}=|A|^{c''}.\] Therefore, by Theorem \ref{thm-borenstein-croot} there exists a subset $A'\su A$ of size $|A'|\ge |A|^{1-\eps}\ge |A|^{1/2}$ with
\[|\ell' A'|\le |A'|^{c''(1+\eps\ell')}\le |A'|^{c'},\]
where the second inequality follows from (\ref{eq-choice-eps}).

Note that
\[|A'|\ge |A|^{1/2}\ge D(c,p)^{1/2}\cdot (C(c))^{n/2}\ge D(c',p)\cdot (C(c'))^{n}.\]
This means that all assumptions are satisfied in Theorem \ref{thm-technical} for $c'$ (which was our induction hypothesis), recalling our choice of $\ell'$ and our assumption that $p$ is large enough. Thus, by applying Theorem \ref{thm-technical} for $c'$, we can conclude that  $A'$ contains $p$ distinct vectors $x_1,\dots,x_p\in A'$ with $x_1+\dots+x_p=0$. As $A'\su A$, this means in particular that $A$ contains $p$ such vectors.

\textbf{Case 2: $\boldsymbol{\sum_{y\in Y} N_y< |A|^{\ell-1-\delta}}$.} In this case, we have
\[\sum_{y\in Y} N_y< |A|^{\ell-1-\delta}\le \frac{|A|^{\ell-1}}{2^{\ell+1}\cdot p^{2\ell}\cdot (C'_{1/\ell})^{\ell n}},\]
using that $|A|\ge D(c,p)\cdot (C(c))^{n}\ge 2^{(\ell+1)/\delta}\cdot p^{2\ell/\delta}\cdot (C'_{1/\ell})^{\ell n/\delta}$. Recalling that $M_y\le |A|$ for all $y\in \Fpn$, this implies
\[\sum_{y\in Y} M_yN_y\le |A|\cdot \sum_{y\in Y} N_y\le \frac{|A|^{\ell}}{2^{\ell+1}\cdot p^{2\ell}\cdot (C'_{1/\ell})^{\ell n}}.\]
On the other hand, by the definition of $Y$, we have
\[\sum_{y\in \Fpn\setminus Y} M_yN_y\le \frac{|A|}{2^{\ell+1}\cdot p^{2\ell}\cdot (C'_{1/\ell})^{\ell n}}\cdot \sum_{y\in \Fpn\setminus Y} N_y\le \frac{|A|}{2^{\ell+1}\cdot p^{2\ell}\cdot (C'_{1/\ell})^{\ell n}}\cdot |A|^{\ell-1}=\frac{|A|^\ell}{2^{\ell+1}\cdot p^{2\ell}\cdot (C'_{1/\ell})^{\ell n}},\]
where the second inequality follows from (\ref{eq-bound-sum-N}). Thus, the number of bad subsets $\{x_1,\dots,x_\ell\}\su A$ is at most
\begin{equation}\label{eq-number-bad-subsets}
\sum_{y\in \Fpn} M_yN_y=\sum_{y\in Y} M_yN_y+\sum_{y\in \Fpn\setminus Y} M_yN_y\le \frac{|A|^\ell}{2^{\ell}\cdot p^{2\ell}\cdot (C'_{1/\ell})^{\ell n}}.
\end{equation}
Let
\[q=\frac{2p^2\cdot (C'_{1/\ell})^n}{|A|}, \]
and note that $0\le q\le 1$, since $|A|\ge D(c,p)\cdot (C(c))^{n}\ge 2^{(\ell+1)/\delta}\cdot p^{2\ell/\delta}\cdot (C'_{1/\ell})^{\ell n/\delta}\ge 2p^2\cdot (C'_{1/\ell})^n$.

Let us now consider a random subset $A^*\su A$ obtained by including every vector in $A$ into the subset $A^*$ with probability $q$, independently for all vectors in $A$. Then we have $\EE[|A^*|]=q\cdot |A|=2p^2\cdot (C'_{1/\ell})^n$.

Recall that every bad subset $\{x_1,\dots,x_\ell\}\su A$ consists of $\ell$ distinct vectors in $A$, and that we bounded the number of bad subsets in (\ref{eq-number-bad-subsets}). For each such bad subset $\{x_1,\dots,x_\ell\}\su A$, the probability of having $x_1,\dots,x_\ell\in A^*$ is $q^\ell$. Let $Y_\text{bad}$ be the number of bad subsets $\{x_1,\dots,x_\ell\}\su A$  with $x_1,\dots,x_\ell\in A^*$, then
\[\EE[Y_\text{bad}]\le q^\ell\cdot \frac{|A|^\ell}{2^{\ell}\cdot p^{2\ell}\cdot (C'_{1/\ell})^{\ell n}}=\frac{2^{\ell}p^{2\ell}\cdot (C'_{1/\ell})^{\ell n}}{|A|^{\ell}}\cdot \frac{|A|^\ell}{2^{\ell}\cdot p^{2\ell}\cdot (C'_{1/\ell})^{\ell n}}=1.\]

Hence
\[\EE\big[|A^*|-Y_\text{bad}\big]\ge 2p^2\cdot (C'_{1/\ell})^n - 1> p^2\cdot (C'_{1/\ell})^n.\]

Thus, there exists some outcome of the random subset $A^*\su A$ such that we have $|A^*|-Y_\text{bad}> p^2\cdot (C'_{1/\ell})^n$. For each of the $Y_\text{bad}$ bad subsets $\{x_1,\dots,x_\ell\}\su A$  with $x_1,\dots,x_\ell\in A^*$, let us now delete one of the elements $x_1,\dots,x_\ell\in A^*$ from the set $A^*$, and let $A'\su A^*$ be the set obtained this way. Then $A'\su A$ is a subset of size $|A'|\ge |A^*|-Y_\text{bad}> p^2\cdot (C'_{1/\ell})^n$ and there does not exist any  bad subset $\{x_1,\dots,x_\ell\}\su A$  with $x_1,\dots,x_\ell\in A'$.

Applying Proposition \ref{prop-repetitions-solutions} with $\lambda=1/\ell$ to the set $A'\su \Fpn$, we can find vectors $y_1,\dots,y_p\in A'$ with $y_1+\dots+y_p=0$ such that every vector in $\Fpn$ appears among $y_1,\dots,y_p$ at most $p/\ell$ times. Let $t$ and $r$ be non-negative integers such that $p=t\ell+r$ and $0\le r\le \ell-1$. In other words, this means that $t=\lfloor p/\ell\rfloor$ and $r=p-\lfloor p/\ell\rfloor\cdot \ell$.

Let us now consider a colouring of the set $\{1,\dots,p\}$ where the colours correspond to the different vectors appearing among $y_1,\dots,y_p$. In other words, two indices $i,j\in\{1,\dots,p\}$ receive the same colour if and only if $y_i=y_j$. Then every colour appears at most $p/\ell$ times on the set $\{1,\dots,p\}$, and so by Lemma \ref{lem-rainbow-sets}, there exists a partition of $\{1,\dots,p\}=S_1\cup\dots\cup S_{t+1}$ with $|S_j|=\ell$ for $j=1,\dots,t$ and $|S_{t+1}|=r$ such that for each $j=1,\dots,t+1$ all elements of $S_j$ have distinct colours. So for each $j=1,\dots,t+1$, the vectors $y_i$ with $i\in S_j$ are distinct.

Recalling that $y_1,\dots,y_p\in A'$ and $A'$ does not contain any bad subset, we can conclude that for each $j=1,\dots,t$, the set $\{y_i\mid i\in S_j\}$ is not bad. Hence, for $j=1,\dots,t$, the sum $\sum_{i\in S_j} y_i$ is not bad (since the vectors $y_i$ for $i\in S_j$ are distinct and $|S_j|=\ell$).

Let us now construct the desired distinct vectors $x_1,\dots,x_p\in A$ with $x_1+\dots+x_p=0$. We start by defining $x_i=y_i$ for $i\in S_{t+1}$ (recall that the vectors $y_i$ for $i\in S_{t+1}$ are distinct). Note that clearly $\sum_{i\in S_{t+1}} x_i=\sum_{i\in S_{t+1}} y_i$.

Now, for $j=1,\dots,t$, let us step by step replace the vectors $y_i$ for $i\in S_j$ with vectors $x_i\in A$ such that $\sum_{i\in S_j} x_i=\sum_{i\in S_j} y_i$. For each index $j=1,\dots,t$, recall that the sum $\sum_{i\in S_j} y_i$ is not bad. This means that the family of all  subsets $\{x'_1,\dots,x'_\ell\}\su A$ with distinct elements $x'_1,\dots,x'_\ell\in A$ satisfying $x'_1+\dots+x'_\ell=\sum_{i\in S_j} y_i$ does not have a cover of size at most $p$. Since we have chosen at most $p$ different vectors $x_i$ throughout our process so far, there must exist a subsets $\{x'_1,\dots,x'_\ell\}\su A$ with distinct elements $x'_1,\dots,x'_\ell\in A$ satisfying $x'_1+\dots+x'_\ell=\sum_{i\in S_j} y_i$, such that $\{x'_1,\dots,x'_\ell\}$ is disjoint from the set of all our previously chosen vectors $x_i$ for $i\in S_1\cup \dots\cup S_{j-1}\cup S_{t+1}$. Let us now assign the vectors $x_i$ for $i\in S_j$ to be $x'_1,\dots,x'_\ell$ (in arbitrary order). Then we have $\sum_{i\in S_j} x_i=x'_1+\dots+x'_\ell=\sum_{i\in S_j} y_i$, and the vectors $x_i$ for $i\in S_j$ are distinct and are also distinct from all the vectors $x_i$ for $i\in S_1\cup \dots\cup S_{j-1}\cup S_{t+1}$.

Continuing this process step by step for all $j=1,\dots,t$, in the end we obtain distinct vectors $x_i\in A$ for all $i\in S_1\cup\dots\cup S_{t+1}=\{1,\dots,p\}$ such that $\sum_{i\in S_j} x_i=\sum_{i\in S_j} y_i$ for $j=1,\dots,t+1$. In other words, $x_1,\dots,x_p\in A$ are distinct vectors, and we have
\[x_1+\dots+x_p=\sum_{j=1}^{t+1}\ \sum_{i\in S_j} x_i=\sum_{j=1}^{t+1}\ \sum_{i\in S_j} y_i=y_1+\dots+y_p=0.\]
This finishes the proof.

\end{document}